\documentclass[a4paper,11pt]{amsart}

\usepackage{amsmath,amssymb,amsthm,graphicx,stmaryrd}
\usepackage{verbatim,enumitem,xfrac}

\usepackage{hyperref}

\newcommand{\CC}{\mathbb{C}}
\newcommand{\EE}{\mathbb{E}}

\newcommand{\PP}{\mathbb{P}}

\newcommand{\VV}{\mathbb{V}}

\newcommand{\ZZ}{\mathbb{Z}}

\newcommand{\cC}{\mathcal{C}}

\newcommand{\cG}{\mathcal{G}} 

\newcommand{\cL}{\mathcal{L}} 
\newcommand{\cH}{\mathcal{H}}

\newcommand{\La}{\Lambda} 
\newcommand{\la}{\lambda}
\renewcommand{\a}{\alpha}
\newcommand{\D}{\Delta} 
\renewcommand{\d}{\delta}

\newcommand{\g}{\gamma} 
 
\renewcommand{\b}{\beta}

\newcommand{\om}{\omega} 
 
\newcommand{\s}{\sigma}

\newcommand{\eps}{\varepsilon}

\newcommand{\el}{\langle} 
\newcommand{\er}{\rangle}
\newcommand{\tr}{\mathrm{tr}}

\newcommand{\lra}{\leftrightarrow}

\renewcommand{\b}{\beta}
\newcommand{\oo}{\infty}

\newcommand{\sm}{\setminus}

\newcommand{\es}{\varnothing}
\newcommand{\se}{\subseteq}

\newcommand{\crit}{\mathrm{c}}

\renewcommand{\r}{\mathrm{r}}
\newcommand{\w}{\mathrm{w}}
\newcommand{\m}{\mathrm{m}}
\newcommand{\id}{\mathrm{id}}

\newcommand{\one}{\hbox{\rm 1\kern-.27em I}}

\newcommand{\be}{\begin{equation}}
\newcommand{\ee}{\end{equation}}

\newtheoremstyle{slthm}
     {}
     {\baselineskip}
     {\slshape}
     {\parindent}
     {\scshape}
     {.}
     { }
     {}

\theoremstyle{slthm}
\newtheorem{definition}{Definition}[section]
\newtheorem{theorem}[definition]{Theorem}
\newtheorem{proposition}[definition]{Proposition}
\newtheorem{lemma}[definition]{Lemma}

\newtheorem{remark}[definition]{Remark}

\title[Large cycles in random permutations]
{Large cycles in random permutations 
related to the Heisenberg model}
\dedicatory{Dedicated to Svante Janson on the occasion of his 60'th birthday}
\author{J. E. Bj\"ornberg}
\date{\today}
\thanks{Research supported by the Knut and Alice
Wallenberg Foundation}

\begin{document}

\begin{abstract}
We study the weighted version of the interchange process
where a permutation receives weight 
$\theta^{\#\mathrm{cycles}}$.  For $\theta=2$ this
is T\'oth's representation of the quantum Heisenberg
ferromagnet on the complete graph.  
We prove, for $\theta>1$, that large cycles appear at
`low temperature'.
\end{abstract}

\maketitle

\section{Introduction}

The interchange process and related models of random
permutations  are interesting
both for their beautiful mathematics,
and  for their  relevance to quantum theoretical
models for magnetization.
The interchange process may be described as follows.  Fix
an integer $n$ and put $n$ labelled balls into
$n$ labelled boxes, ball $i$ in box $i$.  At each
time $t=1,2,\dotsc$, select uniformly 
(independently) a pair of
distinct boxes $i,j$ and transpose the balls inside
them.  At a given time $t$, box $i$ contains
some ball $\pi_t(i)$, where $\pi_t$ is a permutation
of $1,2,\dotsc,n$.
Said otherwise, $\pi_t$ is the composition of $t$
independent, uniformly chosen transpositions.

Being a permuation, $\pi_t$ can be written as a
product of disjoint cycles.   
Schramm showed in~\cite{schramm},
proving a conjecture of Aldous in~\cite{bd},
 that if $t$
is of the form $\lfloor cn\rfloor$ with $c>1/2$,
then with probability approaching 1 as $n\to\oo$,
the largest cycle has size of order $n$
(for $c<1/2$ it is of order $\log n$).
He also described the scaling limit
of the cycles 
in terms of the 
Poisson--Dirichlet distribution. 

In this paper we study random permutations which
are biased towards having many small cycles.
A precise definition  is given in the
next subsection, but roughly speaking we consider the 
weighted version of
the interchange process where each permutation
$\pi$ receives a weight $\theta^{\ell(\pi)}$.
Here $\theta\geq1$ is fixed, and 
 $\ell(\pi)$ is the total number of cycles in $\pi$.
For $\theta=1$ one recovers the interchange process.
Our main result (Theorem~\ref{cycles-thm})
is that large cycles appear for $c>\theta/2$.

The model is motivated by considerations in statistical
physics, where it provides a probabilistic
representation of the (ferromagnetic) quantum
Heisenberg model on the complete graph $K_n$.
It is a notorious open problem to prove that 
the quantum Heisenberg ferromagnet on the 
lattice $\ZZ^d$, $d\geq3$, can exhibit
a nonzero magnetization.

\subsection{Model and main result}
\label{model-sect}

We will in fact work in continuous time, and
with a different time-scaling than 
described above.  
Let $G=K_n=(V,E)$ be the complete graph on the
vertex set $V=\{1,\dotsc,n\}$, with edge set $E=\binom{V}{2}$.
Let $\b>0$ and let $\PP_1(\cdot)$ denote a probability measure
governing a collection $\om=(\om_{xy}:xy\in E)$
of independent rate 1 Poisson processes on $[0,\b]$,
indexed by the edges.  If $\om_{xy}$ has an event at time $t\in[0,\b]$
we  write $(xy,t)\in\om$.  We think of such an event as a
transposition of the vertices $x,y$ at time $t$.  The time-ordered
product of these transpositions gives a permutation $\pi=\pi(\om)$ of
$V$.  More precisely, if we write $(x_iy_i,t_i)$, $i=1,\dotsc,N$,
for the points of $\om$, indexed so that $0<t_1<\dotsc<t_N<\b$,
and write $\tau_i=(x_i,y_i)$ for the transposition of $x_i$ and $y_i$,
then we have that $\pi=\tau_N\dotsb\tau_2\tau_1$.  

Let $\ell=\ell(\om)$ denote the number of cycles in a disjoint-cycle
decomposition of $\pi(\om)$, including singletons.  
Let $\cC_1(\pi),\dotsc,\cC_\ell(\pi)$ denote the cycles ordered by
decreasing size (breaking ties by any rule).
For $\theta\geq1$
we will consider the distribution of $\om$ and $\pi(\om)$ under the
probability measure $\PP_\theta(\cdot)$ given by
\be\label{PP-def}
\frac{d\PP_\theta}{d\PP_1}(\om)=\frac{\theta^{\ell(\om)}}{Z}.
\ee
Here $Z=Z(\theta,\b)$ is the appropriate normalization.
(The same definition makes sense on 
a general finite graph $G$.)

For $\theta=1$ our model is the continuous-time version of the
interchange process,   sped up by a factor $\binom{n}{2}$ compared to
the introduction, viewed at time $\b$.   We take $\b$ of the form
$\b=\la/n$ where $\la>0$ is constant.

\begin{theorem}\label{cycles-thm}
Let $G=K_n$ and $\b=\la/n$ with $\la>\theta>1$.  
For each $\eps>0$ there is
$\d>0$ such that, for $n$ large enough,
\[
\PP_\theta(|\cC_1(\pi)|\geq \d n)\geq 1-\eps.
\]
\end{theorem}

Our proof of Theorem~\ref{cycles-thm}
relies on a colouring-lemma inspired by the
approach of Bollob\'as, Grimmett and 
Janson~\cite{bgj} to the random-cluster model.
Roughly speaking, if we
sample  $\om$ and then 
colour each cycle red or white independently, 
with probability $\sfrac1\theta$
for red, then the conditional distribution
of the red cycles is determined by an interchange 
process.   See Lemma~\ref{twist-lem} for a 
precise statement.
When $\la>\theta$, we can use the results of
Schramm~\cite{schramm} 
to show that there are  red cycles of order $n$.

With high probability 
there are no cycles having size of order $n$ when $\la$
is small enough, e.g.\ when $\la<e^{-1}$.
This follows 
from~\cite[Theorem~6.1]{guw}.
One would expect that there is a critical value
$\la_\crit(\theta)$ such that there are  cycles
of order $n$ for $\la>\la_\crit$ but not for
$\la<\la_\crit$. Schramm's result shows that $\la_\crit(1)=1$.
It follows from the work of
T\'oth~\cite{toth-bec} that  $\la_\crit(2)=2$.
For other values of $\theta$ the existence of $\la_\crit$ is not
known, and Theorem~\ref{cycles-thm} is the first result on the 
occurrence of large cycles in this generality.

Regarding other choices for $G$,
the interchange process ($\theta=1$) has 
been investigated on general graphs by 
Alon and Kozma~\cite{alon-kozma}, and on infinite trees by
Angel~\cite{angel}
and by Hammond~\cite{hammond-infinite,hammond-sharp}.
In ongoing work, Koteck\'y,
Mi{\l}o\'s and Ueltschi
are investigating cases with $\theta\neq1$
on  the hypercube.

\subsection{Relation to the Heisenberg model}
\label{heis-sect}

For $\theta=2$ the cycles 
in our model represent correlations
in the (ferromagnetic, quantum) Heisenberg model,
as shown by T\'oth~\cite{toth}.  
Here is a brief account, see 
the review~\cite{guw} for more details.

The Heisenberg model
on $G$ is given by the Hamiltonian
\[
H=-2\sum_{xy\in E} 
(\s_x^{(1)}\s_y^{(1)}+\s_x^{(2)}\s_y^{(2)}+\s_x^{(3)}\s_y^{(3)}).
\]
Here
\[
\s^{(1)}=\tfrac12\begin{pmatrix} 0& 1\\ 1& 0 \end{pmatrix},\quad
\s^{(2)}=\tfrac12\begin{pmatrix} 0 & -i\\ i& 0 \end{pmatrix},\quad
\s^{(3)}=\tfrac12\begin{pmatrix} 1& 0\\ 0& -1 \end{pmatrix}
\]
are spin-$\tfrac12$ Pauli matrices,
and $\s_x^{(j)}$ acts on the Hilbert space
$\cH_V=\bigotimes_{x\in V}\CC^2$ as 
$\s^{(j)}\otimes\mathrm{Id}_{V\sm\{x\}}$.
Magnetic correlations between vertices $x,y\in G$
are given by
the \emph{correlation functions}
\be\label{corr}
\el \s_x^{(3)}\s_y^{(3)} \er:=
\frac{\tr\big( \s_x^{(3)}\s_y^{(3)} e^{-\b H}\big)}
{\tr \big(e^{-\b H}\big)},
\ee
where $\tr(\cdot)$ denotes the trace of a matrix.
In this formulation the parameter $\b>0$ is usually
called the \emph{inverse temperature}.
(It is the same $\b$ as in Section~\ref{model-sect}.)

T\'oth's representation expresses the 
correlations~\eqref{corr}
probabilistically.  Write $\{x\lra y\}$
for the event that 
$x$ and $y$ belong to the same 
cycle in the permutation $\pi(\om)$.  
Then we have, with $\theta=2$:
\[
\el \s_x^{(3)}\s_y^{(3)} \er=
\tfrac14\PP_2(x \lra y).
\]
Thus the occurrence of large cycles in $\pi(\om)$
corresponds, in physical terms, to magnetic ordering.

The quantum model also possesses 
other probabilistic representations.
In the paper~\cite{toth-bec} T\'oth studied a \emph{lattice gas}
representation and explicitly computed the free energy.
(The same result was independently obtained by 
Penrose~\cite{penrose}.)
By standard arguments one may deduce from these results 
that the quantum Heisenberg model on $G=K_n$
undergoes a phase transition at $\b=2/n$, 
as mentioned above.

\subsection*{Outline and notation}

We describe a graphical representation, 
and present the key colouring-lemma, in
Section~\ref{red-sect}, followed by the proof of 
Theorem~\ref{cycles-thm} in Section~\ref{macro-sect}.

The abbreviation i.i.d.\ means 
independent and identically distributed.
The identity permutation is denoted $\id$.
Unspecified limits are for $n\to\oo$.
If $a_n/b_n\to0$ then we write 
$a_n=o(b_n)$, if $a_n/b_n$ is bounded above 
then we  write $a_n=O(b_n)$.
The indicator of an event $A$ is written
$\one_A$, and takes the value 1 if $A$ happens,
otherwise 0.  Expectation with respect to $\PP_\theta$ will be
written $\EE_\theta$.

\section{Colouring-lemma}
\label{red-sect}

The following graphical representation of a sample $\om$ will be
useful.  
We picture $\om$ in $G\times[0,\b]$, representing a point
$(xy,t)\in\om$ by a `cross' as in Figure~\ref{loops-fig}.
The crosses decompose $V\times[0,\b]$ into a collection $\cL(\om)$ of
disjoint \emph{loops}, obtained as follows.  Starting at some point
$(x,0)\in V\times[0,\b]$, 
we follow the interval $\{x\}\times[0,\b]$ until we reach the
first cross (if any), i.e.\ the first point $(xy,t)\in\om$ for some
$y$.  At this point we jump to $(y,t)$ and
continue as before on the interval $\{y\}\times[t,\b]$ until the next
cross (if any).  Eventually we reach a point $(z,\b)$.  We then
continue the loop at $(z,0)$, i.e.\ we apply periodic boundary
conditions `vertically'.  The loop is completed when we return to the
starting point $(x,0)$.
Denote the loop by $\g$ and identify $\g$ with
the union of all intervals of the form $\{y\}\times[s,t)$
that it traverses.

Recall the permutation $\pi=\pi(\om)$ defined
in Section~\ref{model-sect}.
The points $z$ such that $\g$ visits
$(z,0)$ are precisely $x,\pi(x),\pi^2(x),\dotsc$ 
(in the same order).  Thus the loops $\g$ are in one-to-one
correspondence with the cycles $\cC$ of $\pi$, and the total vertical
length of a loop equals $\b$ times the size of the corresponding cycle.

\begin{figure}
\centering
\includegraphics[scale=.8]{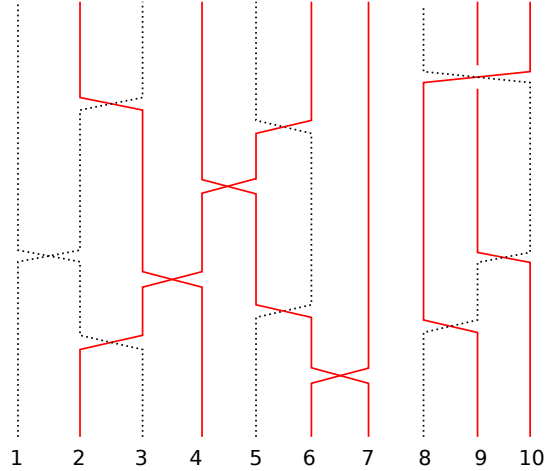}
\caption{
Loop-representation of a sample $\om$.
The vertex set $V=\{1,\dotsc,10\}$ is drawn on the 
horizontal axis, and `time' is on the vertical axis.
Transpositions $(xy,t)\in\om$ are represented 
as crosses.  The resulting loops are
coloured red (solid red lines) or white (dotted lines).
Here $\pi(\om)=(1,3)(2,6,7,4)(9,10)$, and there are
two red loops corresponding to the cycles 
$(2,6,7,4)$ and $(9,10)$.
}
\label{loops-fig}
\end{figure}

Let $\theta>1$ and $r=\sfrac1\theta$.
 Given $\om$, colour each loop $\g\in\cL(\om)$
red or white, independently of each other, with probability
$r$ for red.  Write $R$ and $W$ for the unions of the
red and white loops, respectively;  they are
subsets of $V\times[0,\b]$.  
See Figure~\ref{loops-fig}.
The points of $\om$ (i.e.\ the crosses) now fall into
three categories:  \emph{red}, \emph{white} and \emph{mixed}.
Write $\om_\r$, $\om_\w$ and $\om_\m$ for the red, white and mixed 
crosses, respectively.  Thus $\om=\om_\r\cup\om_\w\cup\om_\m$.

For $H\se E\times[0,\b]$ measurable, let 
$\PP_1^H(\cdot)$ denote the
law of the restriction $\om\cap H$ of $\om$ to $H$.  
Similarly,  for measurable $S\se V\times[0,\b]$, let
$\PP_1^S(\cdot)$ be the restriction of $\PP_1(\cdot)$ to the set
\[
\{(xy,t)\in E\times[0,\b]: (x,t)\in S \mbox{ and } (y,t)\in S \};
\]
that is, the set of points in $E\times[0,\b]$
`between' points of $S$.
The following is the key colouring-lemma.

\begin{lemma}\label{red-lem}
Given $R$, the distribution of $\om_\r$ is $\PP_1^R(\cdot)$,
and $\om_\r$ is conditionally independent of $\om_\w$.
\end{lemma}

In words, conditional on the red set, the red crosses simply
form a Poisson process.
This means that (given $R$) the red cycles are obtained
from a sample of the interchange process, in a way which
will be made precise in Lemma~\ref{twist-lem}.
Lemma~\ref{red-lem} holds for general graphs $G$, with the same
proof. 

One way to check 
Lemma~\ref{red-lem} is to finely discretize the 
Poisson processes. 
We present instead a `clean' proof, and
the basic approach is as follows.  
We write a coloured loop-configuration 
as a pair $(q,\om)$, where $q\in\{\r,\w\}^V$ is a colouring of the
vertices.  We interpret $q_x$ as the colour of the loop containing
$(x,0)$, 
thus we require the pair $(q,\om)$ to be \emph{consistent} in 
that $q_x=q_y$ whenever $x,y$ belong to the same cycle.  
Write $\cC(q)$ for the set of $\om$ that are
consistent with $q$.  Note that, for $\om\in\cC(q)$, the red
and white sets $R$, $W$ are determined by the pair $(q,\om_\m)$ where
$\om_\m$ are the mixed crosses as before.  Indeed, deleting red or
white crosses does not change $R$ or $W$:
compare Figure~\ref{loops-fig-2} with Figure~\ref{loops-fig}.
Thus we may write
$R=R(q,\om_\m)$ and $W=W(q,\om_\m)$, and moreover there is some
freedom in the choice of $\om_\r\cup\om_\w$.  The only restriction on
$\om_\r\cup\om_\w$ is that it is a subset of
\[\begin{split}
H=H(q,\om_\m)=&\{(xy,t): (x,t)\in R\mbox{ and } (y,t)\in R\}\\
&\cup \{(xy,t): (x,t)\in W\mbox{ and } (y,t)\in W\}.
\end{split}\]
We will proceed by conditioning on the mixed crosses $\om_\m$.  
When integrating over the allowed choices for $\om_\r$, a cancellation
occurs in the factor $\theta^{\ell(\om)}$ which removes the
dependencies in $R$.

\begin{proof}[Proof of Lemma~\ref{red-lem}]
Let $X=X(R)$ be a bounded $R$-measurable random variable, 
and consider events
$A=A(\om_\r)$ and 
$B=B(\om_\w)$ depending on $\om_\r$ and $\om_\w$,
respectively.  We will give 
an expression for  $\EE_\theta[X\one_A\one_B]$ which will allow us to
deduce the result.

Write $\ell_R(\om_\r)$ and $\ell_W(\om_\w)$ for the number of red and
white loops, respectively, so that 
$\ell(\om)=\ell_R(\om_\r)+\ell_W(\om_\w)$.
For each colouring $q$, let $\Xi(q)$ be the set of possible $\om_\m$ 
for $\om\in\cC(q)$.  
Then $\om$ is consistent with $q$ if and only if it can
be decomposed as a disjoint union $\om=\xi\cup\zeta$ with
$\xi\in\Xi(q)$ and $\zeta\se H(q,\xi)$.
We thus have that
\be\label{cond-exp}\begin{split}
\EE_\theta[X\one_A\one_B]
&=\frac{1}{Z}\sum_{q\in\{\r,\w\}^V}
\EE_1 \big[\one\{\om\in\cC(q)\}
\theta^{\ell(\om)}  
r^{\ell_R(\om_\r)}(1-r)^{\ell_W(\om_\w)} 
X \one_{A}\one_B\big]\\
&=\frac{1}{Z} \sum_{q\in\{\r,\w\}^V} \EE_1\Big[
\sum_{\xi\se\om} \one\{\xi\in\Xi(q); \om\sm\xi\se H\}   
 (\theta-1)^{\ell_W(\om_\w)} 
X \one_{A}\one_B\Big].
\end{split}\ee
The expectation on the
right-hand-side of~\eqref{cond-exp} can be written in the form 
\be\label{pp-id}
\int d\PP_1(\om)\sum_{\xi\se\om} \a(\xi;\om\sm\xi)
=e^{\b |E|}\int d\PP_1(\xi)\int d\PP_1(\zeta) \a(\xi;\zeta)
\ee
with
\[
\a(\xi;\zeta)=
\one\{\xi\in\Xi(q); \zeta\se H(q,\xi)\}   
 (\theta-1)^{\ell_W(\zeta_\w)} 
X(R(q,\xi)) \one_{A}(\zeta_\r)\one_B(\zeta_\w).
\]
We give a proof of
the identity~\eqref{pp-id} below.
From~\eqref{pp-id} we see that
\[
\EE_\theta[X\one_A\one_B]=\frac{e^{\b |E|}}{Z}
\sum_{q\in\{\r,\w\}^V}
\int_{\Xi(q)} d\PP_1(\xi)  
\int d\PP_1(\zeta)
\one\{\zeta\se H\}
(\theta-1)^{\ell_W(\zeta_\w)} 
X \one_A\one_B.
\]
Here, $\xi$ plays the role of $\om_\m$ and 
$\zeta=\zeta_\r\cup\zeta_\w$ that of $\om\sm\om_\m$.
Thus $X$ is a function of $R=R(q,\xi)$, i.e.\ it is $\xi$-measurable,
and $A$, $B$ depend on $\zeta_\r$ and $\zeta_\w$, respectively.
It follows that
\[
\begin{split}
&\int_{\Xi(q)} d\PP_1(\xi)  
\int d\PP_1(\zeta)
\one\{\zeta\se H\}
(\theta-1)^{\ell_W(\zeta_\w)} 
X \one_A\one_B\\
&=\int_{\Xi(q)} d\PP_1(\xi)  X(R(q,\xi))
e^{-|H^\mathrm{c}|}
\PP_1^R(A)
\EE_1^W[(\theta-1)^{\ell_W(\zeta_\w)}  \one_B(\zeta_\w)],
\end{split}\]
where $H^\mathrm{c}=(E\times[0,\b])\sm H$.
It follows that $\PP_\theta(A\mid R)=\PP_1^R(A)$, that 
$\PP_\theta(B\mid R)=\PP_{(\theta-1)}^W(B)$, and that $A$ and $B$ are
conditionally independent, as claimed.

We now verify~\eqref{pp-id}.
Let $U_1,U_2,\dotsc$ and $U_1',U_2',\dotsc$ denote independent
collections of i.i.d.\ uniform random variables on $E\times[0,\b]$.  
The conditional distribution $\PP_1(\cdot\mid |\om|=m)$ coincides with
the law of $\{U_1,\dotsc,U_m\}$.  Thus
\begin{align*}
&\int d\PP_1(\om)\sum_{\xi\se\om} \a(\xi;\om\sm\xi)
=\sum_{m\geq 0} \PP_1(|\om|=m)
\EE\Big[\sum_{\xi\se\{U_1,\dotsc,U_m\}}
\a(\xi;\{U_1,\dotsc,U_m\}\sm\xi)\Big]\\
&=\sum_{m\geq 0} \frac{e^{-\b|E|}(\b|E|)^m}{m!}
\sum_{k=0}^m \binom{m}{k} 
\EE\big[\a(\{U_1,\dotsc,U_k\};\{U_{k+1},\dotsc,U_m\})\big]\\
&=e^{-\b|E|}\sum_{k\geq 0}\sum_{\ell\geq0}
\frac{(\b|E|)^k}{k!}\frac{(\b|E|)^{\ell}}{\ell!}
\EE\big[\a(\{U_1,\dotsc,U_k\};\{U'_{1},\dotsc,U'_{\ell}\})\big]\\
&=e^{\b|E|}\int d\PP_1(\xi)\int d\PP_1(\zeta)
\a(\xi;\zeta),\quad
\mbox{ as claimed.} \qedhere
\end{align*}
\end{proof}

\begin{figure}
\centering
\includegraphics[scale=.8]{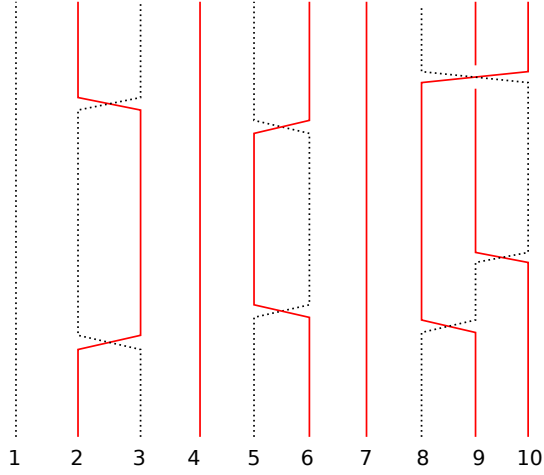}
\caption{
The  configuration $\tilde\om=\om_\m$ 
corresponding to $\om$ of Figure~\ref{loops-fig}.
Here $\tilde\varphi$ is the permutation $(9,10)$
of $R_0=\{2,4,6,7,9,10\}$.
}
\label{loops-fig-2}
\end{figure}
As noted above, $R$ is the same
if we remove all red and all white crosses, that is if
we replace $\om=\om_\r\cup\om_\w\cup\om_\m$
with $\tilde\om=\om_\m$.
We now introduce some notation, see Figure~\ref{loops-fig-2} 
again.  Let 
\[
R_0=\{x\in V: (x,0)\in R\}\se V
\]
be the set of red vertices at time $t=0$.
These are the elements of the red cycles of 
$\pi(\om)$.  For each $x\in R_0$, the trajectory of
$x$ formed by following the vertical lines and crosses
in $\tilde\om$ in the time interval $[0,\b]$
resembles a `crooked line'.  We write 
$\tilde h_t(x)\in V$ for the location of this line
at time $t$.  We write $h_t(x)$ for the corresponding
location obtained using $\om$.  
We take both these functions to be right-continuous in $t$.

The functions 
$h_t(\cdot)$ and $\tilde h_t(\cdot)$ will differ in general,
due to the red crosses.  Our present goal is to describe
their relationship precisely.  Of particular importance
are the functions
\be
\varphi(\cdot):= h_\b(\cdot)\quad\mbox{ and }\quad
\tilde\varphi(\cdot):= \tilde h_\b(\cdot).
\ee
These are  permutations of $R_0$,
and $\varphi$ is precisely the restriction of
$\pi(\om)$ to $R_0$. Hence the cycles of
$\varphi$ are the red cycles of $\pi$.
Clearly $\tilde\varphi$ is  $R$-measurable.

Let $\xi=(\xi_{xy}:x,y\in R_0,x\neq y)$
be a collection of independent rate 1 Poisson processes
on $[0,\b]$, independent of everything else.  
We interpret the points $(xy,t)\in\xi$
as transpositions of the vertices $x,y$ as before, and let
$\s_t$ be the time-ordered product of these transpositions
up to time $t$.  Thus $\s_t$ is a sample of the interchange
process in the set $R_0$.
We will use Lemma~\ref{red-lem} to prove the following.

\begin{lemma}\label{twist-lem}
Given $R$, the conditional distribution of 
$(h_t(\cdot):t\in[0,\b])$ coincides with that of
$(\tilde h_t\circ \s_t:t\in[0,\b])$.
In particular, 
\[
\varphi\overset{(\mathrm{d})}{=} \tilde\varphi\circ\s_\b.
\]
\end{lemma}
In words, 
the conditional
distribution of the red cycles of $\pi(\om)$,
given the union $R$ of the red \emph{loops}, is given by the
interchange process $\s_\b$ on $R_0$, composed with
a `twist' $\tilde\varphi$ (which is a function of $R$).

\begin{proof}[Proof of Lemma~\ref{twist-lem}]
The key observation
is that, thanks to Lemma~\ref{red-lem}
and the symmetry of the complete graph, 
$\om_\r$ has the  same (conditional)
distribution as the collection of points of the 
form
\be\label{red-xi}
(\tilde h_t(x)\tilde h_t(y),t)\mbox{ for }
(xy,t)\in \xi.
\ee
For simpler notation we identify $\om_\r$
with the collection of points in~\eqref{red-xi}.
With this identification, we can
prove the statement of the lemma with equality
(not just in distribution).

Some further notation is required.  
Let $R_t=h_t(R_0)=\tilde h_t(R_0)$ be the set of red vertices
at time $t$.  Let $t_1<t_2<\dotsb$ denote the sequence of 
times at which there  are mixed crosses (elements of 
$\om_\m=\tilde\om$).  Also set $t_0=0$.  Then $R_t$
is constant for $t_{k-1}\leq t<t_k$, for $k\geq 1$. 
Moreover, there is a unique $a_k\in R_{t_{k-1}}$ and a
unique $b_k\not\in R_{t_{k-1}}$ such that 
$R_{t_k}=\psi_k(R_{t_{k-1}})$, where
\[
\psi_k(x)=\left\{\begin{array}{ll}
x, & \mbox{if } x\neq a_k,\\
b_k, &  \mbox{if } x= a_k.\end{array}\right.
\]
Then for $t_k\leq t<t_{k+1}$ we have that
\[
\tilde h_t=\psi_k\circ \psi_{k-1}\circ\dotsb\circ\psi_1.
\]
Now, for all $k\geq 0$ let 
\[
t_k<s^{(k)}_1<s^{(k)}_2<\dotsb<s^{(k)}_{m_k}<t_{k+1}
\]
denote the times of events (transpositions) in $\xi$,
and for $1\leq q\leq m_k$ let 
$\tau^{(k)}_q=(x^{(k)}_q,y^{(k)}_q)$ denote the
corresponding transposition.  Here 
$x^{(k)}_q,y^{(k)}_q\in R_{t_k}$.

With this notation in hand, we can turn to the proof.
For $t=0$ we have $\tilde h_0=\id$, $\s_0=\id$ and
$h_0=\id$, so clearly the claim holds then.  
In fact, for $t<t_1$ we have
that $\tilde h_t=\id$ and  that $h_t=\s_t$, so the claim
holds for such $t$ (by Lemma~\ref{red-lem}).  
Since the functions involved only change at times $t_k$
and $s^{(k)}_q$, we can proceed by induction.

Assume that the claim holds for some $t>0$, i.e.\ assume that
\be\label{time-t}
h_t=\tilde h_t\circ\s_t.
\ee
Let $t'$ be the time
of the next `event'.  That is, either $t'=t_k$ for some $k$, or
$t'=s^{(k)}_q$ for some $k,q$.   It suffices to show that
$h_{t'}=\tilde h_{t'}\circ\s_{t'}$ holds in both cases.

\emph{First case:} $t'=t_k$.  Then $h_{t_k}$ is  obtained by
applying $\psi_k$, thus
\[
h_{t_k}=\psi_k\circ h_t=(\psi_k\circ\tilde h_t)\circ\s_t
=\tilde h_{t_k}\circ\s_t =\tilde h_{t_k}\circ\s_{t_k}.
\]
Here we used~\eqref{time-t} and the fact that $\s_{t_k}=\s_t$.  
Thus $h_{t'}=\tilde
h_{t'}\circ\s_{t'}$ holds in this case.

\emph{Second case:} $t'=s^{(k)}_q$.
Now we have that $\tilde h_{t'}=\tilde h_t$, and that 
\[
\s_{t'}=\tau^{(k)}_q\circ\s_t=
(x^{(k)}_q ,y^{(k)}_q)\circ\s_t.
\]
By the identification of $\om_\r$ with~\eqref{red-xi}, 
we obtain $h_{t'}$ by transposing
$\tilde h_t(x^{(k)}_q)$ and $\tilde h_t(y^{(k)}_q)$.  That is,
\[
h_{t'}(z)=\left\{\begin{array}{ll}
\tilde h_t(y^{(k)}_q), & \mbox{if } h_t(z)=\tilde h_t(x^{(k)}_q),\\
\tilde h_t(x^{(k)}_q), & \mbox{if } h_t(z)=\tilde h_t(y^{(k)}_q),\\
h_t(z), & \mbox{otherwise}.\end{array}\right.
\]
Using that $h_t=\tilde h_t\circ\s_t$, 
we can rewrite this as
$h_{t'}(z)=\tilde h_t(\tau^{(k)}_q(\s_t(z))),$
as required.
\end{proof}

\section{Large cycles}
\label{macro-sect}

In order to use
Lemma~\ref{twist-lem} to analyze the
cycle structure of $\pi$,
we will first need results on the cycle structure
of $\tilde\varphi\circ \s_t$, where $\s_t$
is given by the interchange process and
$\tilde\varphi$ is a non-random permutation.

\subsection{Random and non-random transpositions}

The following result will be obtained using small
modifications of Lemmas~2.1--2.3 of~\cite{schramm}.
(A similar result can be obtained by a small
modification of Theorem~1 of~\cite{berestycki}.)
Here $\s_t$ denotes a sample of the interchange process
($\theta=1$ in~\eqref{PP-def})
on $1,\dotsc,n$, viewed at time $t$, 
and $\tilde\varphi$
is a deterministic permuation of $1,\dotsc,n$.
In this subsection we write $\PP$ for $\PP_1$.

\begin{proposition}\label{intch-prop}
Let $\la>1$ and $t=\la/n$.  For each 
$\eps>0$ there is $\d>0$ and $n_0(\la,\eps,\d)$ such that for 
$n\geq n_0$ we have
\[\PP(|\cC_1(\tilde\varphi\circ\s_t)|\geq \d n)\geq 1-\eps.\]
\end{proposition}

\begin{proof}
First note that, since $\tilde\varphi\circ\s_t$ and
$\s_t\circ\tilde\varphi$ are conjugate, it is equivalent
to consider the largest cycle in the process
$(\s_t\circ\tilde\varphi)_{t\geq0}$ which starts with the permutation
$\tilde\varphi$ at time $t=0$.  We associate with $\tilde\varphi$ a
graph $\tilde G$
whose connected components coincide (as sets) with the
cycles of $\tilde\varphi$.
One way to do this is to
decompose each cycle of $\tilde\varphi$ as:
\[
(x_1,x_2,\dotsc,x_m)=(x_1,x_2)(x_2,x_3)\dotsb
(x_{m-1},x_m)
\]
and let the edges of $\tilde G$ be
the pairs $\{x_i,x_{i+1}\}$ obtained in this way.
For  $t\geq 0$ we let $G_t$ be the (multi-)graph 
obtained by representing each new transposition that
appears in the process $(\s_t)_{t\geq0}$ by an edge,
and we let $\tilde G_t$ be the (multi-)graph obtained
by superimposing $G_t$ on $\tilde G$.
Note that $G_t$ has the distribution of an Erd\H os--R\'enyi
graph $\cG(n,p)$ with $p=1-e^{-t}$.
Also note that each cycle of $\s_t\circ\tilde\varphi$
is contained in some connected component of $\tilde G_t$.

We have that Lemmas~2.1--2.3 of~\cite{schramm} hold in
this situation, with the graph $G_t$ replaced by $\tilde G_t$.
Indeed, one need only check the proof of Lemma~2.2, 
since Lemmas~2.1 and~2.3 do
not refer to the graph.  We provide an outline
of the arguments.

As in Lemma~2.1 of~\cite{schramm}, each time we apply a 
new transposition in $\s_t$, the probability that it splits
an existing cycle so that at least one of the resulting
 cycles has size $\leq k$ is at most $2k/(n-1)$.
Let $V^t_{\tilde G}(k)\se V$ be the union of all
components of $\tilde G_t$ of size at least $k$, 
and let $V^t_{X}(k)\se V$ be the union of all the cycles of 
$\s_t\circ\tilde\varphi$ of size at least $k$.
As in Lemma~2.2 of~\cite{schramm} we have that
\be\label{sl22}
\EE|V^t_{\tilde G}(k)\sm V^t_{X}(k)|\leq 
t\binom{n}{2}\frac{4 k^2}{n-1}.
\ee
(Recall that our process is sped up by a factor
$\binom{n}{2}$ compared to~\cite{schramm}.)
This is because each cycle of size $<k$ which
lies in a component of size $\geq k$ can be associated with
a transposition which split an existing cycle
so that at least one resulting cycle had size 
$\leq k$ (and this can be done so that at most 
two cycles get mapped to the same transposition).
Here we use the fact that at time $t=0$ the components
are equal to the cycles.

Set $t_1=\la/n$ and let $t_0\leq t_1$
with $t_1-t_0=o(n^{-1})$.  
From~\eqref{sl22} we see that, at time $t_0$,
only a tiny fraction of vertices in components
$\geq 2n^{1/4}$ lie in cycles $<2n^{1/4}$:
\[
\EE|V^{t_0}_{\tilde G}(2n^{1/4})\sm 
    V^{t_0}_{X}(2n^{1/4})|\leq  32 \la n^{1/2}.
\]
Let $z(\la)>0$ be the survival probability of a
Galton--Watson process with Poisson($\la$)
offspring distribution, and set $\d'=z(\la)/4$.
Applying the Erd\H os--R\'enyi theorem
to $G_{t_0}\se \tilde G_{t_0}$  we find that
\be\label{sl222}
\begin{split}
\PP(|V^{t_0}_{X}(2n^{1/4})|\leq \d' n)&\leq 
\PP(|V^{t_0}_{\tilde G}(2n^{1/4})\sm V^{t_0}_{X}(n^{1/4})|>\d' n)+o(1)\\
&\leq \frac{32\la n^{1/2}}{\d' n}+o(1).
\end{split}\ee
Here the $o(1)$ is uniform in $\tilde\varphi$.

Lemma~2.3 of~\cite{schramm} applies in our setting,
with adjusted time-scaling.  
Let $j$ be such that $n^{1/4}\leq 2^j\leq 2n^{1/4}$,
write $\rho=2^j/n$, set
\be
\D=\binom{n}{2}^{-1}
\lceil 2^6(\d')^{-1}\rho^{-1}\log_2 (\rho^{-1})\rceil
=O(n^{-5/4}\log n),
\ee
and let $t_0=t_1-\D$.   The Lemma tells us that a large 
fraction of the vertices that lie in cycles
of size $\geq 2n^{1/4}$ at time $t_0$ will lie in cycles of
size of the order $n$ at time $t_1$.  More precisely,
there is a constant $C>0$ such that for all $\eps'\in(0,1/8)$
we have 
\be\label{sl23}\begin{split}
&\PP\left(|V_X^{t_0}(2n^{1/4})\sm V_X^{t_1}(\eps' \d' n)|>\d' n
    \left|\; |V^{t_0}_{X}(2n^{1/4})|> \d' n\right.\right)\\
&\qquad\leq C \cdot (\d')^{-2}\eps' |\log(\eps'\d')|.
\end{split}\ee
Letting $\d=\eps'\d'$ for $\eps'>0$  small enough we find 
using~\eqref{sl222} and~\eqref{sl23} that, for large enough 
$n$,
\[\begin{split}
\PP(V_X^{t_1}(\d n)=\es)&\leq 
\PP(|V_X^{t_0}(2n^{1/4})\sm V_X^{t_1}(\eps' \d' n)|>\d' n)\\
 &\qquad+\PP(|V^{t_0}_{X}(2n^{1/4})|\leq \d' n)\\
&\leq \eps,
\end{split}\]
as claimed.
\end{proof}

\subsection{Proof of Theorem~\ref{cycles-thm}}

Fix $\la>\theta$ and let $\b=\la/n$.
It suffices to prove that, for any sequence
$\d=\d_n\to0$, we have that 
$\PP(|\cC_1(\pi)|<\d n)\to0$,
where $\PP=\PP_\theta$.

Let $N=|R_0|$ denote the number of red vertices, and
let $\La=\la N/n$ so that the auxiliary 
process $\s_t$
of Lemma~\ref{twist-lem} is an interchange process  on $N$
points viewed at time $\La/N$.  
Let $\cC_1^\r=\cC_1(\tilde\varphi\circ\s_{\La/N})$ 
be the largest red cycle.
Pick $\a$ such that $\tfrac1\la<\a<r$.
Since $\cC_1$ is red with probability $r$, 
\be\label{pf-eq}
\begin{split}
r\PP(|\cC_1|<\d n)&\leq 
  \PP(N\geq \a n, |\cC_1^\r|<\d n)+
  \PP(N<\a n, |\cC_1|<\d n)\\
& \leq \EE[\one\{\La\geq\a\la\}\PP(|\cC_1^\r|<(\la\d) n\mid R)]\\
&\qquad   
+ \PP(N<\a n, |\cC_1|<\d n).
\end{split}\ee
When $\La\geq\a\la>1$, Prop.~\ref{intch-prop}
implies that $\PP(|\cC_1^\r|<(\la\d) n\mid R)\to 0$.
The dominated-convergence-theorem therefore implies that 
the first term on the right-hand-side 
of~\eqref{pf-eq} converges to 0.

To bound the second term, let $\PP'$, $\EE'$ and $\VV'$
denote the conditional probability,
expectation and variance, respectively, given the cycle sizes 
$|\cC_1|,|\cC_2|,\dotsc$  Then
$\EE'(N)=\sum_{i\geq1} r |\cC_i|=rn,$
and
\[
\VV'(N)=
\sum_{i\geq1} r(1-r)|\cC_i|^2
\leq |\cC_1|\sum_{i\geq1} |\cC_i|=|\cC_1|n.
\]
Using Chebyshev's inequality: 
\[\begin{split}
\PP(N<\a n,|\cC_1|<\d n)&=
\EE[\one\{|\cC_1|<\d n\}\PP'(N<\a n)]\\
&\leq \EE\Big[\one\{|\cC_1|<\d n\}\frac{|\cC_1|n}{n^2(r-\a)^2}\Big]\\
&\leq \frac{\d}{(r-\a)^2}\to0.
\end{split}
\]
This proves the result.\qed

\begin{remark}
Our proof can straightforwardly
be extended to cases where the cycles $\cC$ 
(or loops $\g$) receive potentially different 
weights.  Indeed,
consider the probability measure $\PP(\cdot)$
given by
\[
\frac{d\PP}{d\PP_1}(\om)=
\frac{1}{Z}\prod_{\g\in\cL(\om)}\theta(\g).
\]
The conclusion of Theorem~\ref{cycles-thm}
holds under the assumption that there 
is a constant $\theta<\la$ such that 
 $1\leq \theta(\cdot)\leq\theta$ uniformly.

Such weights occur for example in the Heisenberg model
in the presence of an external field $h>0$, where
$\theta(\g)=2\cosh(h|\g|)\leq 2\cosh(h\la)$.
See e.g.~\cite{bjo-uel,guw}.
\end{remark}

\subsection*{Acknowledgement}

The author wishes to thank Svante Janson for
helpful discussions about Theorem~\ref{cycles-thm},
Balint T\'oth for pointing out the reference~\cite{toth-bec},
and the anonymous referee for helpful comments.

\end{document}